\documentclass[final,15pt,times]{elsarticle}
\usepackage{amssymb}
\usepackage{amsmath}
\usepackage{times}
\usepackage{orcidlink}
\usepackage{setspace}
\usepackage{graphicx}%
\usepackage{multirow}%
\usepackage{amsmath,amssymb,amsfonts}%
\usepackage{amsthm}%
\usepackage{mathrsfs}%
\usepackage[title]{appendix}%
\usepackage{xcolor}%
\usepackage[T1]{fontenc} % Türkçe karakterler için font kodlaması
\usepackage[utf8]{inputenc} % UTF-8 karakter kodlaması
\usepackage{titlesec}
\usepackage{geometry}
\usepackage{hyperref}
\hypersetup{
    colorlinks=true,
    linkcolor=true,
    filecolor=magenta,      
    urlcolor=cyan,
    pdftitle={Overleaf Example},
    pdfpagemode=FullScreen,
    }
\theoremstyle{thmstyleone}
\newtheorem{theorem}{Theorem}[section]
\newtheorem{proposition}[theorem]{Proposition}

\newtheorem{conclusion}{Conclusion}

\theoremstyle{thmstyletwo}
\newtheorem{example}[theorem]{Example}
\theoremstyle{thmstylethree}
\newtheorem{definition}{Definition}[section]

\def\r{\mathbb R}
\def\E{\mathbb E}

\def\S{\mathbb S}

\def\t{\mathbf t}
\def\n{\mathbf n}
\def\b{\mathbf b}

\def\x{\mathbf x}

\def\u{\mathbf u}
\def\v{\mathbf v}
\usepackage{titlesec}
\journal{}
\titleformat{\section}{\Large\bfseries}{\thesection}{1em}{}
\titleformat{\subsection}{\large\bfseries}{\thesubsection}
{1em}{}

\begin{document}
%%%%%%%%%%%%%%%%%%%%%%%%%%%%%%%%%%%%%%%%%%%%%%%%%%%%%%%%%%%%
\title{\Huge On non-Newtonian Helices in Multiplicative Euclidean Space $\E_*^3$}
%%%%%%%%%%%%%%%%%%%%%%%%%%%%%%%%%%%%%%%%%%%%%%%%%%%%%%%%%%%%
\author{\large
\orcidlink{0000-0000-0000-0000} Aykut HAS$,^{1*}$ \vspace*{0.2cm}
\orcidlink{0000-0000-0000-0000}Beyhan YILMAZ,$^{2}$  \vspace*{0.2cm} \\
$^{1,2}$Department of Mathematics, Faculty of Science, Kahramanmaras Sutcu Imam University, Kahramanmaras, Turkey
\footnote{
\textbf{Corresponding Author:} Aykut HAS \textit{e-mail:} \url{ahas@ksu.edu.tr}\\
\textbf{Contributing Authors:}Beyhan YILMAZ \textit{e-mail:}\url{beyhanyilmaz@ksu.edu.tr};  Muhittin Evren AYDIN \textit{e-mail:}\url{meaydin@firat.edu.tr} } }
\begin{frontmatter}
%%%%%%%%%%%%%%%%%%%%%%%%%%%%%%%%%%%%%%%%%%%%%%%%%%%%%%%%%%%%
\begin{abstract}
%%%%%%%%%%%%%%%%%%%%%%%%%%%%%%%%%%%%%%%%%%%%%%%%%%%%%%%%%%%%
In this article, spherical indicatrices of a curve and helices are re-examined using both the algebraic structure and the geometric structure of non-Newtonian (multiplicative) Euclidean space. Indicatrices of a multiplicative curve on the multiplicative sphere in multiplicative space are obtained. In addition, multiplicative general helix, multiplicative slant helix and multiplicative clad and multiplicative g-clad helix characterizations are provided. Finally, examples and drawings are given.
\end{abstract}

\end{frontmatter}

%%%%%%%%%%%%%%%%%%%%%%%
\section{Introduction}
%%%%%%%%%%%%%%%%%%%%%%%
Classical analysis, which is a widely used mathematical theory today, was defined by Gottfried Leibniz and Isaac Newton in the second half of the 17th century based on the concepts of derivatives and integrals. Constructed upon algebra, trigonometry, and analytic geometry, the classical analysis consists of concepts such as limits, derivatives, integrals, and series. These concepts are regarded as simple versions of addition and subtraction, leading to the designation of this analysis as summational analysis. Classical analysis finds applications in various fields, including natural sciences, computer science, statistics, engineering, economics, business, and medicine, where mathematical modeling is required, and optimal solution methods are sought. However, there are situations in some mathematical models where classical analysis falls short. Therefore, alternative analyses have been defined based on different arithmetic operations while building upon classical analysis. For instance, in 1887, Volterra V. developed an approach known as Volterra-type analysis or multiplicative analysis since it is founded on the multiplication operation \cite{volterra}. In multiplicative analysis, the roles of addition and subtraction operations in classical analysis are assumed by the multiplication and division operations, respectively. Following the definition of Volterra analysis, Grossman M. and Katz R. conducted some new studies between 1972 and 1983. This led to the development of the non-Newtonian analysis, which also involves fundamental definitions and concepts \cite{grossman,grossman2}. These analyses have been referred to as geometric analysis, bigeometric analysis, and anageometric analysis. Multiplicative analysis has emerged as an alternative approach to classical analysis and has become a significant area of research and development in the field of mathematics. These new analyses may allow for a more effective resolution of various problems by examining different mathematical structures. Furthermore, these studies contribute to the expansion of the boundaries of mathematical analysis and find applications in various disciplines.

Arithmetic is an integer field which is a subset of the real numbers. An arithmetic system is the structure obtained by algebraic operations defined in this field. In fact, this field can be considered as a different interpretation of the real number field such that a countable number of infinitely ordered objects can be formed and these structures are equivalent or isomorphic to each other. The generator function, which is used to create arithmetic systems, is a one-to-one and bijective transformation whose domain is real numbers and whose value set is a subset of positive real numbers. The unit function $I$ and the function $e^x$ are examples of generator functions. Just as each generator produces a single arithmetic, each arithmetic can be produced with the help of a single generator. Multiplicative analysis has its own multiplicative space. In this special space, the classical number system has turned into a multiplicative number system consisting of positive real numbers, denoted by $\r_*$. Likewise, the basic mathematical operations in classical analysis have also turned into their purely multiplicative versions. This is clearly shown in the table below.
\begin{equation}
\begin{tabular}{|l|l|l|}
	\hline
	$a+_*b$  & $e^{\log a+\log b}$ & $ab$ \\
	\hline
	$a-_*b$ & $e^{\log a-\log b}$ & $\frac{a}{b}$ \\
	\hline
	$a\cdot_*b$ & $e^{\log a\log b}$ & $a^{\log b}$ \\
	\hline
	$a/_*b$ & $e^{\log a/\log b}$ & $a^{\frac{1}{\log b}},~ b\neq 1$ \\
	\hline
\end{tabular} \label{m}
\end{equation}
\centerline{\textbf{Table 1.} Basic multiplicative operations.}
Multiplicative analysis, contrast to not a completely new topic, has recently started to be explored and discovered more in today's context. The main reason behind this lies in the successful modeling of problems that cannot be addressed using classical analysis, achieved through the application of multiplicative analysis. This characteristic has led many mathematicians to prefer multiplicative analysis for solving challenging problems that are otherwise difficult to model within their respective fields. Stanley D. took the lead in this regard and re-announced geometric analysis as multiplicative analysis \cite{stanley}. On this subject, fractal growths of fatigue defects in materials are studied by Rybaczuk M. and Stoppel P. \cite{rybaczuk} and the physical and fractional dimension concepts are studied by Rybaczuk M. and Zielinski W. \cite{rybaczuk2}. In addition, there are many studies on multiplicative analysis in the field of pure mathematics. For example, the non-Newtonian efforts in complex analysis are \cite{ali,bashirov2}, in numerical analysis \cite{yazici,boruah,aniszewska}, in differential equations  \cite{bashirov3,yalcin,waseem}. Also, Bashirov et al. reconsider multiplicative analysis with some basic definitions, theorems, propositions, properties and examples \cite{bashirov}. The multiplicative Dirac system and multiplicative time scale are studied by Y. Emrah et al. \citep{emrah1,emrah2,emrah3}.

Georgiev S. brought a completely different perspective to multiplicative analysis with the books titled {\it "Multiplicative Differential Calculus"}, {\it "Multiplicative Differential Geometry"} and {\it "Multiplicative Analytic Geometry"} published in 2022 \cite{svetlin,svetlin2,svetlin3}. Unlike previous studies, Georgiev S. used operations as purely multiplicative operations and almost reconstructed the multiplicative space. These books have been recorded as the initial studies in particular for multiplicative geometry. Georgiev S.'s book \cite{svetlin2} serves as a guide for researchers in this field by encompassing numerous fundamental definitions and theorems pertaining to curves, surfaces, and manifolds. The book elucidates how to associate basic geometric objects such as curves, surfaces, and manifolds with multiplicative analysis, shedding light on their properties in multiplicative spaces. Additionally, it emphasizes the connections between multiplicative geometry and other mathematical domains, making it a valuable resource for researchers working in various branches of mathematics. Afterward, Nurkan S.K. et al. tried to construct geometry with geometric calculus. In addition, Gram-Schmidt vectors are obtained \cite{karacan}. On the other hand, Aydın M.E. et al. studied rectifying curves in multiplicative Euclidean space. The multiplicative rectifying curves are fully classified and visualized through multiplicative spherical curves \cite{evren}.

A helix curve is the curve that a point follows as it rotates around a fixed axis in a three-dimensional space. The helix curve is formed as a result of this rotational movement, and the rotation time around the axis determines the stability of the curve. While the helix curve is important in terms of geometry, it is also increasing in different branches of science. For example, helix is a term used for the connections of DNA. The double helix structure of DNA is called an image helix \cite{bio}. In computer graphics and 3D applications, helix curves are used in sections of complex surfaces and their results \cite{bil}. The helix is used in blades and aerospace engineering for the design and performance analysis of propellers and rotor blades \cite{hav}. In addition, helices have been traditionally studied by many researchers with their different properties \cite{izumuya,takahashi,aykut,mahmut,kaya}.

 In this study, spherical indicatrices and helix curves, which are important for differential geometry, are examined in multiplicative space. Spherical indicatrices, general helix, slant helix, clad helix and g-clad helix are rearranged with reference to multiplicative operations. Moreover, in the multiplicative Euclidean space, basic concepts such as orthogonal vectors, orthogonal system, curves, Frenet frame, etc. are mentioned. In addition, it is aimed to make these basic concepts more memorable by visualizing them.

%%%%%%%%%%%%%%%%%%%%%%%
\section{Multiplicative Calculus and Multiplicative Space} \label{sec2}
%%%%%%%%%%%%%%%%%%%%%%%
In this section, basic definitions and concepts regarding multiplicative analysis are presented within the framework of the information provided by Georgiev S. \cite{svetlin,svetlin2,svetlin3}. As mentioned in the Introduction, multiplicative analysis has its own distinct space. Let 
\begin{equation*}
\r_*=\{exp(a):a\in\r\}=\r^+.  
\end{equation*}
Given by equation \eqref{m}, a multiplicative structure is formed by the field $(\r_*, +_*,\cdot_* )$. Each element of the space $\r_*$ is referred to as a multiplicative number and is denoted by $a_*\in\r_*$, where $a_* = exp(a)$. In addition, the unit elements of multiplicative addition and multiplicative multiplication operations are $0_*=1$ and $1_*=e$, respectively. Also, we say that a number $a\in\r_*$ is the multiplicative positive number if $a > 1$ or $a>0_*$. A number $a\in\r_*$ is said to be multiplicative negative if it is not equal to $0_*$ and $a<0_*$ or $0<a<1$.
Inverse elements of multiplicative addition and multiplicative multiplication are, respectively,
$$
-_* a=1/a,\quad a^{-1_*}=e^{\frac{1}{\log a}}.
$$
The multiplicative absolute value function in multiplicative space is given as follows in line with the information above 
\begin{equation*}
 \left \vert a \right \vert_*=\left\{ 
\begin{array}{ll}
a, & a>0_* \\ 
0_*, & a=0_*\\
-_*a,& a<0_*,
\end{array}%
\right.
\end{equation*}
where it can be used as $1/a$ instead of $-_*a$ and $a<0_*$ in stead of $a\in(0,1)$.
Also, some important features of the power function on multiplicative space are as follows for all $a\in\r_*$ and $k\in\mathbb{N}$
$$
a^{k_*}=e^{(\log a)^k}, \quad a^{\frac{1}{2}_*}=\sqrt[*]{a}=e^{\sqrt{\log a}}.
$$

Each ordered coordinate vector $(u_1,u_2,...,u_n)$ of the space $\r_*^n$ determines a point $P$. This point $P$ is called the $n$-dimensional multiplicative vector. The multiplicative addition and multiplicative scalar multiplication are in the multiplicative vector space $\r_*^n$ for each $\u,\v\in\r_*^n$ and $k\in\r_*$, as follows
\begin{eqnarray*}
    \u+_*\v&=&(u_1+_*v_1,...,u_n+_*v_n)=(u_1 v_1,..,u_n v_n),\\
    k\cdot_*\u&=&(k\cdot_*u_1,...,k\cdot_*u_n)=(u_1^{\log k},...,u_n^{\log k})=e^{\log k\log \u}.
\end{eqnarray*}
where $\log\u=(\log u_1,...,\log u_n).$

Let $\u$ and $\v$ be any two multiplicative vectors in the multiplicative vector space $\r_*^n$. The multiplicative inner product of the $\u$ and $\v$ vectors is
\begin{equation*}
 \langle \u,\v\rangle_*=e^{\langle \log\u,\log\v\rangle}. 
\end{equation*}
where $\langle,\rangle$ is the Euclidean inner product. Moreover, if the multiplicative vectors $\u$ and $\v$ are orthogonal to each other in the multiplicative sense a relation can be given as
\begin{equation*}
 \langle \u,\v\rangle_*=0_*. 
\end{equation*}
In Fig. \ref{fig5}, we present the graph of the multiplicative orthogonal vectors.
\begin{figure}[hbtp]
\begin{center}
\includegraphics[width=.4\textwidth]{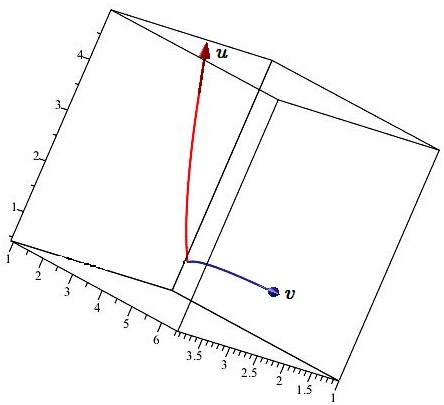}\\
\end{center}
\caption{\text{Multiplicative orthogonal vectors $\u=(e^{\frac{1}{2}},e^{-\frac{3}{4}},e^{\frac{3}{2}})$ and $\v=(e^{\frac{3}{4}},e,e^{\frac{1}{4}})$.} \label{fig5}}
\end{figure}

For a vector $\u\in\r_*^n$, the multiplicative norm of $\u$ is defined as follows,
\begin{equation*}
    \|\u\|_*=e^{\sqrt{\langle \log\u,\log\u\rangle}}.
\end{equation*}
 The  multiplicative cross product of $\u$ and $\v$ in $\r_*^3$ is defined by
\begin{equation*}
 \u \times_* \v = (e^{\log u_2\log v_3 -\log u_3\log v_2},e^{\log u_3\log v_1 -\log u_1\log v_3}, e^{\log u_1\log v_2 -\log u_2\log v_1}) .
\end{equation*}
It is direct to prove that the multiplicative cross product holds the standard algebraic and geometric properties. For example, $\u \times_* \v$ is multiplicative orthogonal to $\u$ and $\v$. In addition, $\u \times_* \v = \mathbf{0}_*$ if and only if $\u$ and $\v$ are multiplicative collinear. In Fig. \ref{fig4}, we present the graph of the multiplicative orthogonal system.
\begin{figure}[hbtp]
\begin{center}
\includegraphics[width=.4\textwidth]{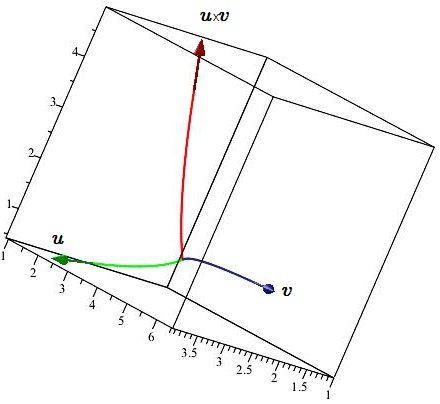}\\
\end{center}
\caption{\text{Multiplicative orthogonal system $\u=(e^{\frac{1}{2}},e^{-\frac{3}{4}},e^{\frac{3}{2}})$, $\v=(e^{\frac{3}{4}},e,e^{\frac{1}{4}})$ and $\u\times_*\v=(e^{-\frac{27}{16}},e,e^{\frac{17}{16}})$.} \label{fig4}}
\end{figure}
\newpage
Let $\u$ and $\v$ represent two unit multiplicative directions in $\r_*^3$. Denote by $\theta$ the multiplicative angle between the multiplicative unit vectors $\u$ and $\v$. Hence
\begin{equation*}
    \theta=\arccos_*(e^{\langle \log\u,\log\v\rangle}).
\end{equation*}
Trigonometric functions in terms of multiplicative space, are
$$
\sin_* \theta=e^{\sin\log \theta}, \quad \cos_* \theta=e^{\cos\log \theta}.
$$
Similarly, trigonometric functions $\tan_* \theta$ and $\cot_* \theta$ are also available at hand. In addition, multiplicative trigonometric functions provide multiplicative trigonometric relations in parallel with classical trigonometric relations. For example, there is the idendity $\sin_*^{2*}\theta+_*\cos_*^{2*}\theta=1_*$. For other relations, see \cite{svetlin}.

Let $x\in I\subset\r_*$ and $f(x) \subset \r_*$ be a multiplicative function. The first multiplicative derivative of $f$ at $x\in I$ is defined as follows, which will be denoted by $f^*(x)$,  
$$
f^*(x)=\lim_{h\rightarrow 0_*}(f(x+_*h)-_*f(x))/_*h.  
$$
In terms of the usual arithmetical operations,
\begin{eqnarray*}
f^*(x)&=&\lim_{h\rightarrow 1}\left(\frac{f(xh)}{f(x)}\right)^{1/\log h}\\
&=&\lim_{h\rightarrow 1}e^{\frac{\log\frac{f(xh)}{f(x)}}{\log h}}.
\end{eqnarray*}
If the L' Hospital's rule applies here, we get
$$
f^*(x)=e^{\frac{xf'(x)}{f(x)}}.
$$
In addition, if the function $f$ is differentiable in the multiplicative sense and continuous, it is called * (multiplicative) differentiable function.

It can be said that the multiplicative derivative provides some properties such as linearity, Leibniz rule and chain rules as in the classical derivative as follows

\begin{enumerate}
	\item  $(a\cdot_*f)^{*}(x)=a\cdot_*f^{*}(x),~ \text{for all}~ a\in\mathbb{R_*}$,
	
	\item $(f(x)\mp_*g(x))^*=f^{*}(x)\mp_*g^{*}(x)$,
	
	\item $(f(x)\cdot_*g(x))^*=f^{*}(x)\cdot_*g(x)+_*g^{*}(x)\cdot_*f(x)$,

    \item $(f(x)/_*g(x))^*=(f^{*}(x)\cdot_*g(x)-_*g^{*}(x)\cdot_*f(x))/_*g^{2_*}(x) $,
	
	\item $(f\circ g)^*(x)=f^{*}(g(x))\cdot_*g^*(x)$,
	
	\item $f^{*^{(k)}}(x)=(f^{*^{(k-1)}}(x))^*, k\in\mathbb{N}, $ 
\end{enumerate}
where $f,g$ are $*$ (multiplicative)-differentiable for each $x\in\r_*$. 

The definition of the multiplicative integral is given as the inverse operator of the multiplicative derivative. The  multiplicative indefinite integral of the function $f(x)$ is defined by 
\begin{equation*}
 \int_* f(x)\cdot_*d_*x=e^{\int\frac{1}{x}\log f(x)dx},\quad x\in\r_*.
\end{equation*}

The geometric location of points with equal multiplicative distances from a point in multiplicative space is called a multiplicative sphere. The equation of the sphere with centered at $C(a,b,c)$ and radius $r$ is
\begin{equation*}
  \|P-_*C\|_*=r,  
\end{equation*}
where $P=(x,y,z)$ is the representation point of the multiplicative sphere, so
\begin{equation*}
e^{ (\log x - \log a)^2+(\log y - \log b)^2+(\log z - \log c)^2}=e^{(\log r)^2}. 
\end{equation*}
In Fig. \ref{figz} we show the multiplicative sphere with centered at multiplicative origin $O(0_*,0_*,0_*)$ and radius $1_*$.
\begin{figure}[hbtp]
\begin{center}
\includegraphics[width=.3\textwidth]{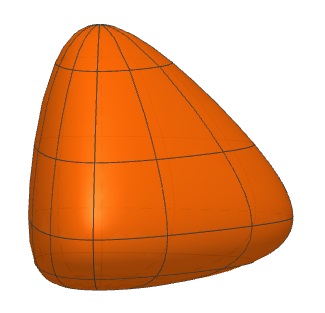}\\
\end{center}
\caption{\text{A multiplicative sphere with centered at multiplicative origin $O(0_*,0_*,0_*)$ and radius $1_*$.} \label{figz}}
\end{figure}
%%%%%%%%%%%%%%%%%%%%%%%
\section{Differential Geometry of Curves in Multiplicative Space} \label{sec2}
%%%%%%%%%%%%%%%%%%%%%%%
A multiplicative parametrization of class $C_*^k$ ($k \geq 1_*$) for a curve $\x$ in $\mathbb{R}_*^3$ (i.e., the component-functions of $\x$ are 
$k$-times continuously multiplicative differentiable), is a multiplicative vector valued function $\x: I \subset \mathbb{R}_* \to \mathbb{E}_*^3$, where $s$ is mapped to $\x(s) = (x_1(s), x_2(s), x_3(s))$. In particular, a parametric multiplicative curve $\x$  is regular if and only if $\| \x^*(s) \|_* \neq 0_*$ for any $s \in I$. Looking at it dynamically, the multiplicative vector $\x^*(s)$ represents the multiplicative velocity of the multiplicative curve at time $s$. For a multiplicative curve $\x$ to have multiplicative naturally parameters, the necessary and sufficient condition is that the curve is from the class $C_*^k$ and $\|\x^*(s)\|_*= 1_*$ for each $s \in I$.

Given $s_0\in I$, the multiplicative arc length of a multiplicative regular parametrized curve $\x(s)$ from the point $s_0$, is by definition 
\begin{equation}
 h(s)=\int_{*s_0}^{s}\|\x^*(t)\|_*\cdot_*d_*t. \label{a}  
\end{equation}
As an example, the multiplicative circle curve in multiplicative plane with center $(0_*,0_*,0_*)$ and radius $r=e^{-2}$ is given by the equation $\x(s)=e^{-2}\cdot_*(e^{\frac{1}{2}}\cos_*2s,e^{\frac{1}{2}}\cdot_*\sin_*2s,e^{\sqrt{3}})$ in $\r_*^3$. It can be plotted as in Fig \ref{fig1}.
\begin{figure}[hbtp]
\begin{center}
\includegraphics[width=.3\textwidth]{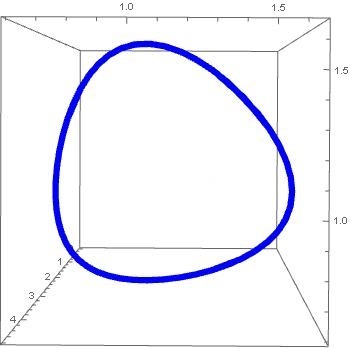}
\end{center}
\caption{A multiplicative circle in the plane $z=e^{\frac{\sqrt{3}}{2}}$ with centered at $(0_*,0_*,0_*)$, radius $r=1/e^2$ and $0_*<s<e^{2\pi}$.} \label{fig1}
\end{figure}

The multiplicative Frenet trihedron of a naturally parameterized multiplicative curve $\x(s)$ are
$$\t(s)=\x^*(s),\quad \quad  \n (s)= \x^{**}(s)/_*\| \x^{**}(s)\|_*,\quad  \quad\b (s)= \t (s) \times_* \n (s).$$
\quad The vector field $\t(s)$ (resp. $\n(s)$ and $\b(s)$) along $\x(s)$ is said to be {\it multiplicative tangent} (resp. {\it multiplicative principal normal} and {\it multiplicative binormal}). It is direct to prove that $\{\t (s), \n (s), \b (s) \}$ is mutually multiplicative orthogonal and $\n (s) \times_* \b (s) =\t (s) $ and $\b (s) \times_* \t (s) =\n (s) $. We also point out that the arc length parameter and multiplicative Frenet frame are independent from the choice of multiplicative parametrization \cite{svetlin2}. 

To give an example, the multiplicative Frenet vectors of the multiplicative curve $$\x(s)=\left((e^3/_*e^5)\cdot_*\cos_* s,(e^3/_*e^5)\cdot_*\sin_* s,e^{4}/_*e^{5}\cdot_*e^{ s}\right)$$ are
\begin{eqnarray*}
    \t(s)&=&(-_*(e^3/_*e^5)\cdot_*\sin_* s,(e^3/_*e^5)\cdot_*\cos_* s,e^4/_*e^5)\\
    \n(s)&=&(-_*\cos_* s,-_*\sin_* s,0_*)\\
    \b(s)&=&((e^4/_*e^5)\cdot_*\sin_*s,-_*(e^4/_*e^5)\cdot_*\cos_*s,e^3/_*e^5)
\end{eqnarray*}
In Fig. \ref{figx}, we present the graph of the multiplicative Frenet frame on the multiplicative curve $\x(s)$.
\begin{figure}[hbtp]
\begin{center}
\includegraphics[width=.33\textwidth]{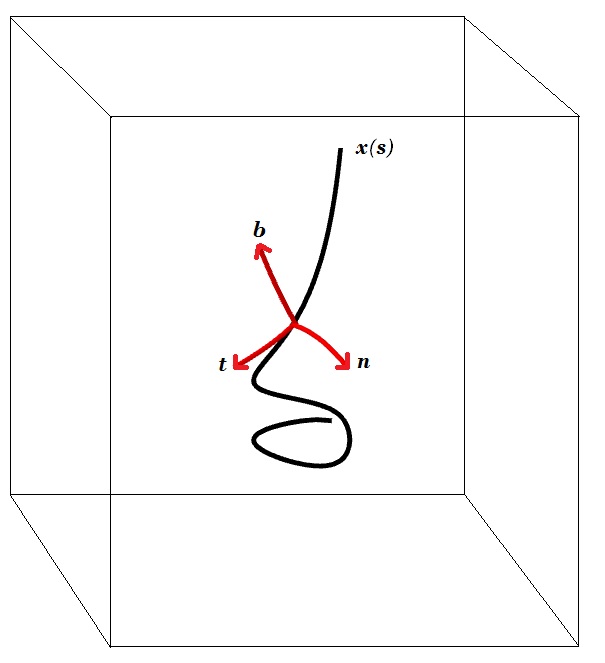}\\
\end{center}
\caption{\text{Multiplicative curve and its multiplicative Frenet frame.} \label{figx}}
\end{figure}

The multiplicative Frenet formulae of $\x$ are given by
\begin{eqnarray*}
    \t^*&=&\kappa \cdot _* \n,\\
    \n^*&=&-_*\kappa \cdot_* \t +_* \tau \cdot_* \b,\\
    \b^*&=&-_* \tau \cdot_* \n,
\end{eqnarray*}
where $\kappa=\kappa(s)$ and $\tau=\tau(s)$ are the curvature and the torsion functions of $\x$, calculated by
\begin{eqnarray}
    \kappa(s)&=&\| \x^{**}(s)\|_*=e^{\langle \log\x^{**},\log\x^{**}\rangle^{\frac{1}{2}}},\\
    \tau(s)&=&\langle \n^*(s),\b(s) \rangle_* =e^{\langle \log\n^*(s),\log\b(s)\rangle} \label{cur}.
\end{eqnarray}

%%%%%%%%%%%%%%%%%%%%%%%
\section{Main Results} \label{sec2}
%%%%%%%%%%%%%%%%%%%%%%%

%%%%%%%%%%%%%%%%%%%%%%%
\subsection{Multiplicative Spherical Indicatries} \label{sec2}
%%%%%%%%%%%%%%%%%%%%%%%
Consider a multiplicative curve $\x(s)\in\r_*^{3}$. The multiplicative Frenet vectors of $\x$ also evolve along the curve as a multiplicative vector field. The thing to note here is that since the multiplicative Frenet vectors of the multiplicative curve $\x$ are multiplicative unit vectors, they form a curve on the multiplicative sphere. In this section, such curves will be examined.

 The multiplicative curve $\x(s)$ is associated with multiplicative Frenet vectors $\{\t,\n,\b\}$. Now, let us consider the unit multiplicative tangent vectors along $\x(s)$. These vectors collectively form another curve, denoted by $\x_{t}=\t$. This new curve resides on the surface of a multiplicative sphere with a radius of $1_*$ and centered at the multiplicative origin $O=(0_*,0_*,0_*)$. The multiplicative curve $\x_{t}$ is often referred to as the multiplicative spherical indicatrix associated with the unit multiplicative tangent vector $\t$. We will call this curve {\it multiplicative tangent indicatrix} of the original multiplicative curve $\x$, in line with the more conventional notation. With similar thought, we will call multiplicative curves $\x_n=\n$ and $\x_b=\b$ as {\it multiplicative normal indicatrix} and {\it multiplicative binormal indicatrix} of $\x$, respectively.
\begin{proposition}
Let $\x_t$ be the multiplicative tangent indicatrix of a multiplicative naturally parameterized curve $\x$. The multiplicative naturally parameter $s_t$ of $\x_t$ is given by
\begin{equation*}
  s_t=e^{\int^s\frac{1}{u}\kappa(u)du},
\end{equation*}
where $s$ is  multiplicative naturally parameter of $\x$ and $\kappa(s)$ is multiplicative curvature of $\x$
\end{proposition}
\begin{proof}
 Let $\x(s)$ be the multiplicative naturally parameterized curve. Also, let $\x_t(s)=\t(s)$ be the multiplicative tangent indicatrix of $\x(s)$. Considering Eq. \eqref{a}, we get the multiplicative naturally parameter of $\x_t$ as follows                   
\begin{equation*}
 s_t=\int_{*}^s\|\t^*(u)\|_*\cdot_*d_*u.  
\end{equation*} 
Then from the definition of multiplicative curvature, we get
\begin{equation*}
 s_t=\int_{*}^s\kappa(u)\cdot_*d_*u  
\end{equation*} 
or equivalently
\begin{equation*}
  s_t=e^{\int^s\frac{1}{u}\kappa(u)du}.
\end{equation*} 
\end{proof}
\begin{theorem}
Let $\x_t$ be the multiplicative tangent indicatrix of a multiplicative naturally parameterized curve $\x$ with $\kappa\neq 0_*$ on $I$. The multiplicative Frenet vectors $\{T_t,N_t,B_t\}$ of $\x_t$ satisfy
\begin{eqnarray*}
 T_t&=&\n, \\
 N_t&=&(-_*\t+_*f\cdot_*\b)/_*e^{(1+(\log f(s))^2)^\frac{1}{2}},\\
 B_t&=&(f\cdot_*\t+_*\b)/_*e^{(1+(\log f(s))^2)^\frac{1}{2}},
\end{eqnarray*}
where $f=f(s)$ and $f=\tau/_*\kappa$.
\end{theorem}
\begin{proof}
Since $\x_t$ is the multiplicative tangent indicatrix of $\x$, we have 
$$
\x_t=\t.
$$
Taking the multiplicative derivative of both sides of the above equation with respect to $s$,
$$
(d_*\x_t/_*d_*s_t)\cdot_*(d_*s_t/_*d_*s)=\kappa\cdot_*\n.
$$
Using {\it Proposition 4.1} and putting $T_t=d_*\x_t/_*d_*s_t$, we get
\begin{equation}
  T_t=\n. \label{b} 
\end{equation}
If we take the multiplicative derivative of Eq. \eqref{b} with respect to $s$ and apply multiplicative Frenet formulas, we have
\begin{equation*}
 \kappa\cdot_*(d_*T_t/_*d_*s_t)=-_*\kappa\cdot_*\t+_*\tau\cdot_*\b
\end{equation*}
and
\begin{equation*}
d_*T_t/_*d_*s_t=-_*\t+_*(\tau/_*\kappa)\cdot_*\b.
\end{equation*}
Considering the multiplicative norm, the following equation is obtained: 
\begin{eqnarray}
\|d_*T_t/_*d_*s_t\|_*&=&e^{(\langle-\log\t,-\log\t\rangle+(\frac{\log\tau}{\log\kappa})^2\langle\log\b,\log\b\rangle)^{\frac{1}{2}}},\notag\\
 &=&e^{\sqrt{(1+(\frac{\log\tau}{\log\kappa})^2)}}.\label{d}
\end{eqnarray}
In that case, we can see that
\begin{equation*}
N_t=(d_*T_t/_*d_*s_t)/_*\|d_*T_t/_*d_*s_t\|_*=(-_*\t+_*(\tau/_*\kappa)\cdot_*\b)/_*e^{(1+(\frac{\log\tau}{\log\kappa})^2)^\frac{1}{2}}.
\end{equation*}
Setting $\tau/_*\kappa=f$,
\begin{equation}
    N_t=(-_*\t+_*f\cdot_*\b)/_*e^{(1+(\log f(s))^2)^\frac{1}{2}}.\label{c}
\end{equation}
On the other hand, if we take into account Eqs. \eqref{b} and \eqref{c} along with the multiplicative Frenet formulas, we obtain the final Frenet vector as
\begin{equation*}
 B_t=[\n\times_*(-_*\t+_*f\cdot_*\b)]/_*e^{(1+(\log f(s))^2)^\frac{1}{2}}.   
\end{equation*}
When we organize the multiplicative operations, we obtain
\begin{equation}
 B_t=(f\cdot_*\t+_*\b)/_*e^{(1+(\log f(s))^2)^\frac{1}{2}}. \label{e}
\end{equation}
\end{proof}
\begin{proposition}
Let $\x_t$ be the multiplicative tangent indicatrix of a multiplicative naturally parameterized curve $\x$ with $\kappa\neq 0_*$ on $I$. The multiplicative curvatures of $\x_t$ are
  \begin{equation*}
     \kappa_t=e^{(1+(\log f(s))^2)^\frac{1}{2}}\quad\text{and}\quad \tau_t=\sigma\cdot_*e^{(1+(\log f(s))^2)^\frac{1}{2}},
  \end{equation*}
where $\sigma=f^*/_*(e^{\log\kappa(1+(\log f(s))^2)^\frac{3}{2}}).$
\end{proposition}
\begin{proof} 
The first equality follows by Eq. \eqref{d},
 \begin{equation*}
     \kappa_t=e^{\langle \log\x_t^{**}, \log\x_t^{**}\rangle^{\frac{1}{2}}}=e^{(1+(\log f(s))^2)^\frac{1}{2}}.
 \end{equation*}
Next considering  calculate the multiplicative derivative of $N_t$ with respect to $s_t$, considering Eq. \eqref{c}. Also, is chosen $e^{(1+(\log f(s))^2)^\frac{1}{2}}=\lambda$ in Eq. \eqref{c}, so
\begin{equation*}
 (d_*N_t/_*d_*s_t)\cdot_*(d_*s_t/_*d_*s)=[(-_*\kappa\cdot_*\n+_*f^*\cdot_*\b-_*f\cdot_*\tau\cdot_*\n)\cdot_*\lambda-_*\lambda\cdot_*(-_*\t+_*f\cdot_\b)]/_*\lambda^{2*}   
\end{equation*}
after this we can write
\begin{equation*}
    N^{*}_t\cdot_*\kappa=[\lambda^{*}\cdot_*\t-_*\lambda\cdot_*(\kappa+_*f\cdot_*\tau)\cdot_*\n+_*(\lambda\cdot_*f^*-_*\lambda^*\cdot_*f)\cdot_*\b]/_*\lambda^{2*}  
\end{equation*}
and so
\begin{equation}
    N^{*}_t=[\lambda^{*}\cdot_*\t-_*\lambda\cdot_*(\kappa+_*f\cdot_*\tau)\cdot_*\n+_*(\lambda\cdot_*f^*-_*\lambda^*\cdot_*f)\cdot_*\b]/_*\lambda^{2*}\cdot_*\kappa. \label{xx} 
\end{equation}
Then from Eqs. \eqref{e} and \eqref{xx}, we obtain
\begin{eqnarray*}
  \tau_t=\langle\log N^*_t,\log B_t\rangle_*&=&(f\cdot_*\lambda^*)/_*\lambda^{3*}\cdot_*\kappa+_*(\lambda\cdot_*f^*-_*\lambda^*\cdot_*f)/_*\lambda^{3*}\cdot_*\kappa\\
  &=&(f^*\cdot_*\lambda)/_*\lambda^{3*}\cdot_*\kappa.
 \end{eqnarray*}
Here again let's consider the choice $e^{(1+(\log f(s))^2)^\frac{1}{2}}=\lambda$, so we get
\begin{equation*}
 [f^*/_*(e^{(1+(\log f(s))^2)^\frac{3}{2}}\cdot_*\kappa)]\cdot_*e^{(1+(\log f(s))^2)^\frac{1}{2}}.  
\end{equation*}
Finally, if a choice is made in the form $\sigma=f^*/_*(e^{\log\kappa(1+(\log f(s))^2)^\frac{3}{2}}\cdot_*\kappa)$, the above-mentioned equation becomes
\begin{equation*}
     \tau_t=\sigma\cdot_*e^{(1+(\log f(s))^2)^\frac{1}{2}}.
 \end{equation*}
 \end{proof}
Using similar arguments, we may have the following results.
\begin{proposition}
Let $\x_n$ be the multiplicative normal indicatrix of a multiplicative naturally parameterized curve $\x$. The multiplicative arc parameter $s_n$ of the multiplicative curve $\x_n$ provides
\begin{equation}
  s_n=\int_*\kappa(s)\cdot_*e^{(1+(\log f(s))^2)^\frac{1}{2}}\cdot_*d_*s  \label{n2}
\end{equation}
where $f=f(s)$ and $f=\tau/_*\kappa$.
\end{proposition}
\begin{theorem}
Let $\x_n$ be the multiplicative normal indicatrix of a multiplicative naturally parameterized curve $\x$. The multiplicative Frenet vectors $\{T_n,N_n,B_n\}$ of $\x_n$ as follows
\begin{eqnarray*}
 T_n&=&(-_*\t+_*f\cdot_*\b)/_*e^{(1+(\log f(s))^2)^\frac{1}{2}}, \\
 N_n&=&(\sigma/_*e^{(1+(\log \sigma(s))^2)^\frac{1}{2}})\cdot_*[((f\cdot_*\t+_*\b)/_*e^{(1+(\log f(s))^2)^\frac{1}{2}})-_*\n/_*\sigma],\\
 B_n&=&(e/_*e^{(1+(\log \sigma(s))^2)^\frac{1}{2}})\cdot_*[((f\cdot_*\t+_*\b)/_*e^{(1+(\log f(s))^2)^\frac{1}{2}}-_*\n\cdot_*\sigma],
\end{eqnarray*}
where $\sigma=f^*/_*(e^{\log\kappa(1+(\log f(s))^2)^\frac{3}{2}}).$
\end{theorem}
\begin{proposition}
Let $\x_n$ be the multiplicative normal indicatrix of the multiplicative curve $\x$. The multiplicative curvatures of the normal indicatrix $\x_n$ are described as follows
\begin{equation}
\kappa_n=e^{(1+(\log \sigma(s))^2)}\quad\text{and}\quad \tau_n=\Gamma\cdot_*e^{(1+(\log f(s))^2) }\label{n1}
\end{equation}
  where $\Gamma=\sigma^*/_*(e^{\log\kappa(1+(\log f(s))^2)(1+(\log \sigma(s))^2)^\frac{3}{2}}.$
\end{proposition}

\begin{proposition}
Let $\x_b$ be the multiplicative binormal indicatrix of a multiplicative naturally parameterized curve $\x$. The multiplicative arc parameter $s_b$ of the multiplicative curve $\x_b$ provides
\begin{equation*}
  s_b=\int_*\tau(s)d_*s.  
\end{equation*}
\end{proposition}
\begin{theorem}
Let $\x_b$ be the multiplicative binormal indicatrix of a multiplicative naturally parameterized curve $\x$. The multiplicative Frenet vectors $\{T_b,N_b,B_b\}$ of $\x_b$ satisfy
\begin{eqnarray*}
 T_b&=&-_*\n, \\
 N_b&=&(\t-_*f\cdot_*\b)/_*e^{(1+(\log f(s))^2)^\frac{1}{2}},\\
 B_b&=&(f\cdot_*\t+_*\b)/_*e^{(1+(\log f(s))^2)^\frac{1}{2}},
\end{eqnarray*}
where $f=f(s)$ and $f=\tau/_*\kappa$.
\end{theorem}
\begin{proposition}
Let the multiplicative curve, denoted as $\x_b$, be the binormal indicatrix of the multiplicative curve $\x$. Then, the multiplicative curvatures of the $\x_b$ are described as follows
  \begin{equation*}
     \kappa_b=e^{(1+(\log f(s))^2)^\frac{1}{2}}/_*f\quad\text{and}\quad \tau_b=(-_*\sigma\cdot_*e^{(1+(\log \sigma(s))^2)^\frac{1}{2}})/_*f
  \end{equation*}
  where $\sigma=f^*/_*(\kappa\cdot_*(e^{1+(\log f(s))^2})^{\frac{3}{2}*}).$
\end{proposition}

\begin{example}
Let $\x$ $:I\subset\mathbb{R_*}\rightarrow \E_*^{3}$ be multiplicative naturally parametrized curve in $\mathbb{R}_*^{3}$  parameterized by%
\begin{equation*}
  \x(s)=\left(e^s,e^{\frac{e^2}{2}},e^{\frac{e^3}{6}}\right).  
\end{equation*}
In Fig. \ref{figa}, we present the graph of the multiplicative spherical indicatrices of $\x$.
\end{example}
\begin{figure}[hbtp]
\begin{center}
\includegraphics[width=.9\textwidth]{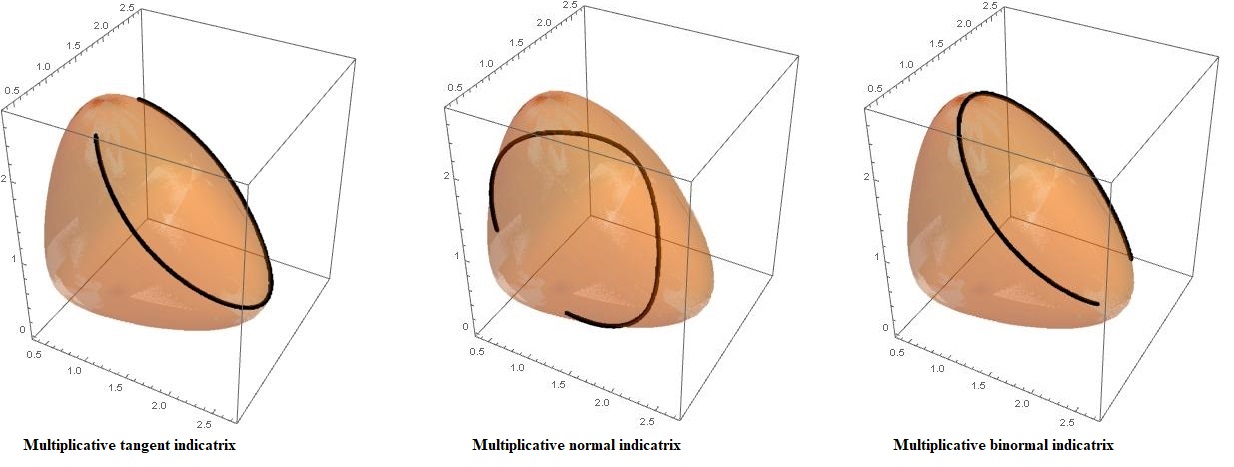}\\
\end{center}
\caption{\text{Multiplicative spherical indicatrices.} \label{figa}}
\end{figure}
%%%%%%%%%%%%%%%%%%%%%%%
\subsection{Multiplicative Helices} \label{sec2}
%%%%%%%%%%%%%%%%%%%%%%%
\begin{definition}
Consider the multiplicative curve $\x: I \subset \mathbb{R}_* \to \mathbb{E}_*^3$ with $\kappa\neq 0_*$. If the multiplicative tangent vector field of the curve $\x$ makes a constant multiplicative angle with a constant multiplicative vector, then the curve $\x$ is referred to as a multiplicative general helix \cite{svetlin2}. 
\end{definition}
\begin{theorem}
Let $\x: I \subset \mathbb{R}_* \to \mathbb{E}_*^3$ be a multiplicative curve with $\kappa\neq 0_*$. The multiplicative space curve $\x$ is a multiplicative general helix if and only if the multiplicative ratio of multiplicative torsion and multiplicative curvature is constant. In other words, it is
 $$\kappa/_*\tau=c,\quad c\in\mathbb{R}_*.$$
\end{theorem}
\begin{proof}
 The proof of the theorem is explained by Georgiev S.(see \cite{svetlin2}).   
\end{proof}
\begin{example}
Let $\x$ $:I\subset\mathbb{R_*}\rightarrow \E_*^{3}$ be multiplicative naturally parametrized general helix curve in $\mathbb{R}_*^{3}$ parameterized by%
\begin{equation*}
\x(s)=\left(e^3/_*e^5\cdot_*\cos_*s,e^3/_*e^5\cdot_*\sin_*s,e^4/_*e^5\cdot_*e^s\right).
\end{equation*}
With the help of multiplicative curvature formulas from Eq. \eqref{cur}, we give
\begin{equation*}
 \kappa(s)=e^3/_*e^5\quad\text{and}\quad \tau=e^4/_*e^5.   
\end{equation*}
Since $$\tau/_*\kappa=e^{(\log e^4/\log e^5)/(\log e^3/\log e^5)}=e^4/_*e^3$$
is a multiplicative constant, $\x$ is a multiplicative helix. In Fig. \ref{fig2}, we present the graph of the multiplicative general helix.
\begin{figure}[hbtp]
\begin{center}
\includegraphics[width=.8\textwidth]{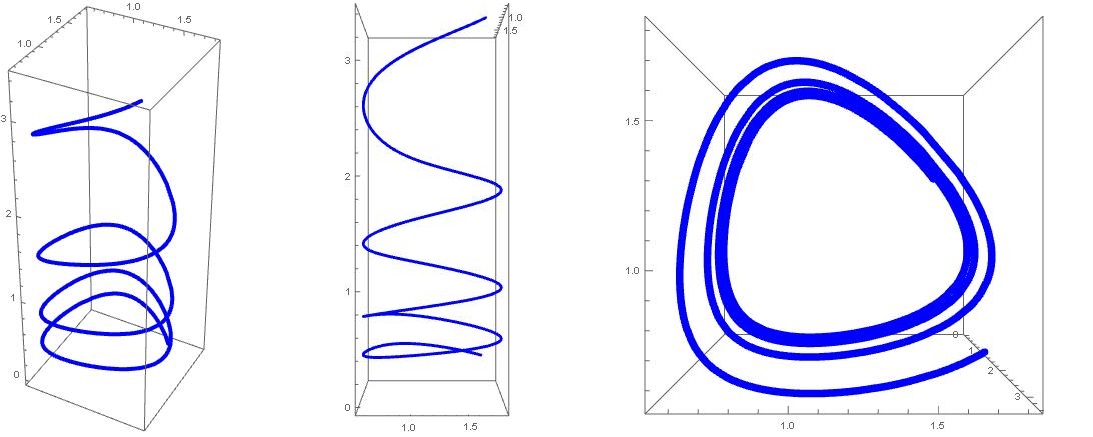}\\
\end{center}
\caption{\text{Multiplicative general helix.} \label{fig2}}
\end{figure}
\end{example}

\begin{definition}
Let $\x: I \subset \mathbb{R}_* \to \mathbb{E}_*^3$ be the multiplicative curve with $\kappa\neq 0_*$. If the multiplicative normal vector field of the curve $\x$ makes a constant multiplicative angle with a constant multiplicative vector, then the curve $\x$ is referred to as a multiplicative slant helix. 
\end{definition}
\begin{theorem}
Let $\x: I \subset \mathbb{R}_* \to \mathbb{E}_*^3$ be multiplicative curve  with $\kappa\neq 0_*$. The multiplicative curve $\x$ is a multiplicative slant helix if and only if the following equality is a multiplicative constant function.
 \begin{equation}
  \sigma(s)=[\kappa^{2*}(s)/_*(\kappa^{2*}(s)+_*\tau^{2*}(s))^{\frac{3}{2}*}]\cdot_*(\tau(s)/_*\kappa(s))^*. \label{n3}
\end{equation}
\end{theorem}
\begin{proof}
 Suppose that multiplicative naturally parametrized curve $s\mapsto \x(s)$ is a {multiplicative slant helix}. Since the multiplicative normal vector field of the multiplicative curve $\x$ makes a constant multiplicative angle with $\v$, which is a constant multiplicative vector, we have
 \begin{equation}
 \langle \n, \v\rangle_*=\cos_*\theta, \label{hel.1}   
\end{equation}
 where $\theta$ constant multiplicative angle. Taking a multiplicative derivative of Eq. \eqref{hel.1}, we get
\begin{equation}
\langle \n^*, \v\rangle_*=0_*, \label{hel.9}   
\end{equation}
and
\begin{equation*}
-_*\kappa\cdot_*\langle \t, \v\rangle_*+_*\tau\cdot_*\langle \b, \v\rangle_*=0_*.  
\end{equation*}
As can be seen from the elements of multiplicative Frenet frame and Eq. \eqref{hel.1}, there is a constant angle between $\n$ and fixed direction $\v$ and there is also a constant angle between $\b$ and fixed direction $\v$. Then the following equations are provided,
\begin{eqnarray}
   \langle \t, \v\rangle_*&=&(c\cdot_*\tau)/_*\kappa, \label{hel.5}\\
   \langle \b, \v\rangle_*&=&c, \textbf{ }c\in\r_*. \label{hel.6} 
\end{eqnarray}
In terms of the multiplicative Frenet frame, we can write the decomposition for $\v$ as
\begin{eqnarray*}
\v&=&e^{\langle\log\t,\log\v\rangle\log\t+\langle\log\n,\log\v\rangle\log\n+\langle\log\b,\log\v\rangle\log\b} \\
&=& \langle \t,\v  \rangle _* \cdot_* \t +_*  \langle \n,\v  \rangle _* \cdot_* \n +_*  \langle  \b,\v  \rangle _* \cdot_* \b.
\end{eqnarray*}
The constant direction $\v$ from Eqs. \eqref{hel.1}, \eqref{hel.5} and \eqref{hel.6} is obtained as follows
\begin{equation}
\v = (c\cdot_*\tau)/_*\kappa \cdot_* \t +_* \cos_*\theta \cdot_* \n +_* c\cdot_* \b.
\end{equation}
Since $\v$ is the multiplicative unit vector, taking the multiplicative norm of both sides of the above equation, we get
\begin{eqnarray*}
e^{\langle \log\v,\log\v\rangle^{\frac{1}{2}}}&=&((c\cdot_*\tau)/_*\kappa)^{2*}\cdot_*e^{\langle \log\t,\log\t\rangle^{\frac{1}{2}}}
+_*\cos_*^{2*}\theta\cdot_* e^{\langle \log\n,\log\n\rangle^{\frac{1}{2}}}\\&+_*&c^{2*}\cdot_*e^{\langle \log\b,\log\b\rangle^{\frac{1}{2}}}
\end{eqnarray*}
or
\begin{equation*}
c^{2*}\cdot_*(\tau^{2*}/_*\kappa^{2*}+_*e)=\sin_*^{2*}\theta.
\end{equation*}
If the necessary algebraic operations are performed here, we obtain
\begin{equation*}
 c=(\kappa/_*(\kappa^{2*}+_*\tau^{2*})^{\frac{1}{2}*}\cdot_*\sin_*\theta.
\end{equation*}
Therefore, we can easily write $\v$ as
\begin{equation}
\v = \tau/_*(\kappa^{2*}+_*\tau^{2*})^{\frac{1}{2}*}\cdot_*sin_*\theta \cdot_* \t +_* \cos_*\theta \cdot_* \n +_* \kappa/_*(\kappa^{2*}+_*\tau^{2*})^{\frac{1}{2}*}\cdot_*sin_*\theta\cdot_* \b. \label{hel.14}
\end{equation}
Take the multiplicative derivative of Eq. \eqref{hel.9}, we get
\begin{equation}
\langle \n^{**}, \v\rangle_*=0_*. \label{hel.15}   
\end{equation}
From multiplicative Frenet frame and Eqs. \eqref{hel.14} and \eqref{hel.15}, we have 
\begin{eqnarray*}
 &&\langle -_*\kappa^*\cdot_*\t-_*(\kappa^{2*}+_*\tau^{2*})\cdot_*\n+_*\tau^*\b,\tau/_*(\kappa^{2*}+_*\tau^{2*})^{\frac{1}{2}*}\cdot_*sin_*\theta \cdot_* \t\\
  &&+_* \cos_*\theta \cdot_* \n +_* \kappa/_*(\kappa^{2*}+_*\tau^{2*})^{\frac{1}{2}*}\cdot_*sin_*\theta\cdot_* \b\rangle_*=0_*.
\end{eqnarray*}
Here the following equation exists
\begin{equation*}
(\kappa\cdot_*\tau^*-\tau\cdot_*\kappa^*)/_*(\kappa^{2*}+_*\tau^{2*})^{\frac{3}{2}*}\cdot_*\tan_*\theta+e=0_*,   
\end{equation*}
and finally, we get
\begin{equation*}
    \tan_*\theta=(\kappa\cdot_*\tau^*-\tau\cdot_*\kappa^*)/_*(\kappa^{2*}+_*\tau^{2*})^{\frac{3}{2}*}.
\end{equation*}
Since the multiplicative angle $\theta$ is constant, after the necessary adjustments, we obtain that
\begin{equation*}
  \kappa^{2*}/_*(\kappa^{2*}+_*\tau^{2*})^{\frac{3}{2}*}\cdot_*(\tau/_*\kappa)^*=c,\quad c\in\r_*. 
\end{equation*}
\end{proof}
\begin{example}
Let $\x$ $:I\subset\mathbb{R_*}\rightarrow \E_*^{3}$ be multiplicative naturally parametrized slant helix curve in $\mathbb{R}_*^{3}$ as%
\begin{equation*}
\x(s)=\left(x_1(s),x_2(s),x_3(s)\right)
\end{equation*}
where
\begin{eqnarray*}
    x_1(s)&=&e^9/_*e^{400}\cdot_*e^{\sin\log 25s}+_*e^{25}/_*e^{144}\cdot_*e^{\sin\log 9s}\\
    x_2(s)&=&-_*e^9/_*e^{400}\cdot_*e^{\cos\log 25s}+_*e^{25}/_*e^{144}\cdot_*e^{\cos\log 9s}\\
    x_3(s)&=&e^{15}/_*e^{136}\cdot_*e^{\sin\log 17s}
\end{eqnarray*}
In Fig. \ref{fig3}, we present the graph of the multiplicative slant helix.
\end{example}
\begin{figure}[hbtp]
\begin{center}
\includegraphics[width=.88\textwidth]{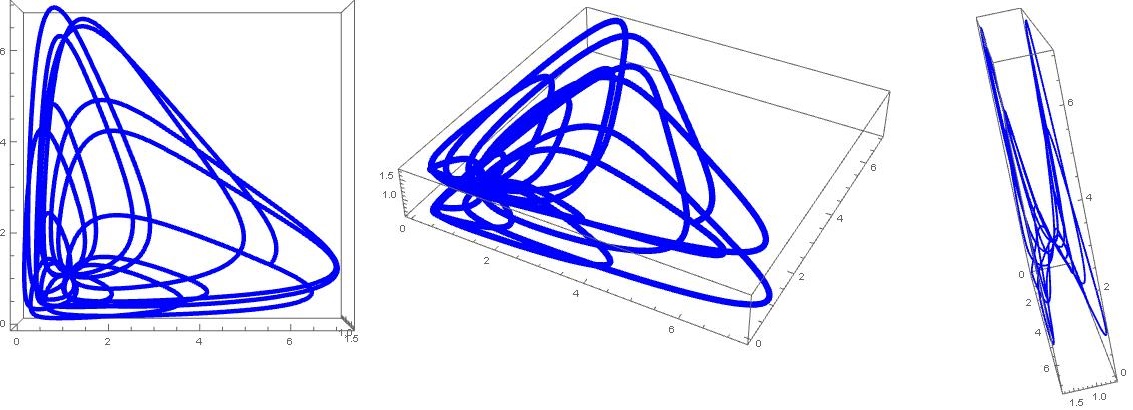}\\
\end{center}
\caption{\text{Multiplicative slant helix.} \label{fig3}}
\end{figure}
\begin{definition}
Let a regular multiplicative curve $\x$ be given in the multiplicative space with $\kappa\neq 0_*$. The multiplicative curve $\x$ is called the multiplicative clad helix if the multiplicative spherical image of the multiplicative principal normal vector $\n: I\rightarrow \S^{2}_*$ ($\S^{2}_*$ denotes the multiplicative sphere) of the curve $\x$ is part of the multiplicative cylindrical helix in $\S^{2}_*$. 

Therefore we remark that a multiplicative slant helix is a multiplicative clad helix. We have the following characterization of clad helices. \label{def1}
\end{definition}
\begin{theorem}
 Let $\x$ be a multiplicative naturally parametrized curve with $\kappa\neq 0_*$. Then $\x$ is a multiplicative clad helix if and only if
 \begin{eqnarray*}
 \Gamma=\sigma^*/_*[\kappa\cdot_*(e+_*f^{2*})\cdot_*(e+_*\sigma^{2*})^{\frac{3}{2}*}]
 \end{eqnarray*}
 is a constant function. \label{teo1}
 Here, $f=\tau/_*\kappa$
 and $\sigma=f^*/_*(\kappa\cdot_*(e+_*f^{2*})^{\frac{3}{2}*}).$
\end{theorem}
\begin{proof}
With the multiplicative normal indicatrix of the multiplicative curve $\x$ being $\x_n$, we know from Eq. \eqref{n1} that the multiplicative curvatures of $\x_n$ are as follows 
$$ \kappa_n=(e+_*\sigma^{2*})^{\frac{1}{2}*},$$
$$\tau_n=\Gamma\cdot_*(e+_*\sigma^{2*})^{\frac{1}{2}*}.$$
It follows that $\Gamma=\tau_n/_*\kappa_n$. For a part of $\x_n$ to be a multiplicative cylindrical helix, $\tau_n/_*\kappa_n$ must be a multiplicative constant. This means that $\Gamma$ is a multiplicative constant.
\end{proof}
\begin{definition}
Let a regular multiplicative curve $\x$ be given in the multiplicative space with $\kappa\neq 0_*$. The multiplicative curve $\x$ is called the multiplicative g-clad helix if the multiplicative spherical image of the multiplicative principal normal vector $\n: I\rightarrow \S^{2}_*$ of the curve $\x$ is part of the multiplicative slant helix in $\S^{2}_*$. 

We have the following characterization of g-clad helices. \label{def2}
\end{definition}
\begin{theorem}
 Let $\x$ be a multiplicative naturally parametrized curve with $\kappa\neq 0_*$. Then $\x$ is a multiplicative g-clad helix if and only if 
 \begin{equation*}
    \psi(s)= \Gamma^{*}(s)/_*\left((\kappa^{2*}(s)+_*\tau^{2*}(s))^{\frac{1}{2}*}\cdot_*(e+_*\sigma^{2*}(s))^{\frac{1}{2}*}\cdot_*(e+_*\Gamma^{2*}(s))^{\frac{3}{2}*}\right)
 \end{equation*}
 is a constant function. \label{teo2}
\end{theorem}
\begin{proof}
With the multiplicative normal indicatrix of the multiplicative curve $\x$ being $\x_n$, from Eq. \eqref{n1} the multiplicative curvatures of $\x_n$ are as follows 
\begin{eqnarray*}
 \kappa_n&=&(e+_*\sigma^{2*})^{\frac{1}{2}*},\\
 \tau_n&=&\Gamma\cdot_*(e+_*\sigma^{2*})^{\frac{1}{2}*}.
\end{eqnarray*}
If the necessary algebraic operations are performed here, we get 
$$\kappa_n^{2*}+_*\tau_n^{2*}=(e+_*\sigma^{2*})\cdot_*(e+_*\Gamma^{2*}).$$
From Eq. \eqref{n3}, we know that
\begin{equation*}
\left(\kappa_n^{2*}/_*(\kappa_n^{2*}+_*\tau_n^{2*})^{\frac{3}{2}*}\right)\cdot_*(\tau_n/_*\kappa_n)^*=c, c\in\r_*.
\end{equation*}
So, we can easily see that
\begin{equation}
\Gamma^{*}/_*\left((\kappa^{2*}+_*\tau^{2*})^{\frac{1}{2}*}\cdot_*(e+_*\sigma^{2*})^{\frac{1}{2}*}\cdot_*(e+_*\Gamma^{2*})^{\frac{3}{2}*}\right). \label{n4}
\end{equation}
If the normal indicatrix of the multiplicative curve $\x$ is a slant helix, Eq. \eqref{n4} is a constant function. This completes the proof.
\end{proof}
\begin{conclusion}
Considering Theorems \eqref{teo1} and \eqref{teo2}  that a multiplicative slant helix is a multiplicative helix with the condition $\sigma\equiv 0_*$, a multiplicative clad helix is a multiplicative slant helix with the condition $\Gamma\equiv 0_*$ and a multiplicative g-clad helix is a multiplicative clad helix with the condition $\psi\equiv 0_*$. Hence, we have the following relation 
\begin{displaymath} \left\{\begin{array}{ll} &\textrm{the family of}\\ &\textrm{multiplicative}\\&\textrm{helices} \end{array}\right\}\subset\left\{\begin{array}{ll} &\textrm{the family of}\\ &\textrm{multiplicative}\\&\textrm{slant helices} \end{array}\right\}\subset\left\{\begin{array}{ll} &\textrm{the family of}\\&\textrm{multiplicative}\\ &\textrm{clad helices}\\ \end{array}\right\}\subset\left\{\begin{array}{ll} &\textrm{the family of}\\&\textrm{multiplicative}\\ &\textrm{g-clad helices}\\ \end{array}\right\}.
\end{displaymath}
\end{conclusion}

\begin{example}
Let $\x$ $:I\subset\mathbb{R_*}\rightarrow \E_*^{3}$ be multiplicative naturally parametrized clad helix curve in $\mathbb{R}_*^{3}$ parameterized by%
\begin{equation*}
\x(s)=\left(x_1(s),x_2(s),x_3(s)\right)
\end{equation*}
where
\begin{eqnarray*}
    x_1(s)&=&e^{18}\cdot_*\cos_*3s\cdot_*\cos_*(e^6\cdot_*\cos_* 3s),\\
    x_2(s)&=&e^{-18}\cdot_*\cos_*3s\cdot_*\sin_*(e^6\cdot_*\cos_* 3s),\\
    x_3(s)&=&\sin_*2s.
\end{eqnarray*}
In Fig. \ref{fig8}, we present the graph of the multiplicative clad helix.
\begin{figure}[hbtp]
\begin{center}
\includegraphics[width=.82\textwidth]{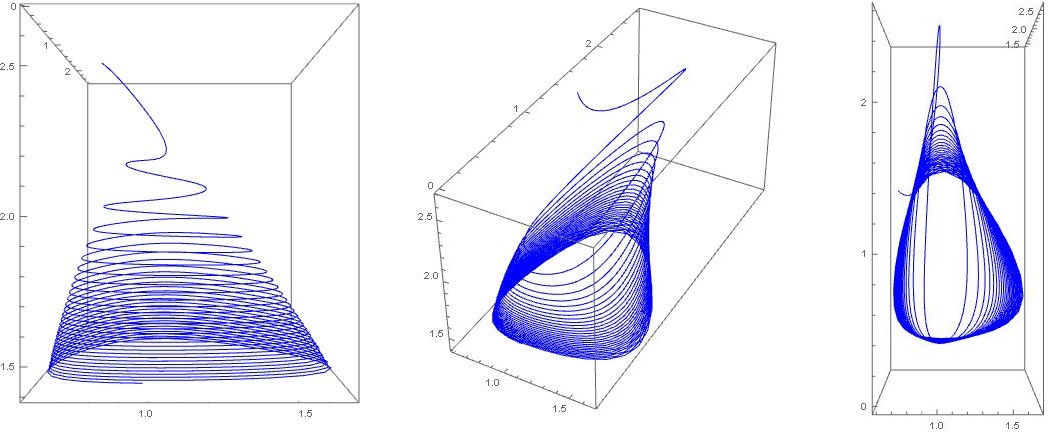}\\
\end{center}
\caption{\text{Multiplicative clad helix.} \label{fig8}}
\end{figure}
\end{example}


\begin{thebibliography}{99}

\bibitem{volterra} Volterra V., Hostinsky B., Operations Infinitesimales Lineares, Herman, Paris, 1938.

\bibitem{grossman} Grossman M., Katz R., Non-Newtonian Calculus, 1st ed., Lee Press, Pigeon Cove Massachussets, 1972.

\bibitem{grossman2} Grossman M., Bigeometric Calculus: A System with a Scale-Free Derivative, Archimedes Foundation, Massachusetts, 1983.

\bibitem{stanley} Stanley D., A multiplicative calculus, Primus, 9(4) (1999), 310-326.

\bibitem{rybaczuk} Rybaczuk M., Stoppel P., The fractal growth of fatigue defects in materials, International Journal of Fracture, 103 (2000), 71-94.

\bibitem{rybaczuk2} Rybaczuk M., Zielinski W.,  The concept of physical and fractal dimension I. The projective dimensions, Chaos, Solitons and Fractals, 12(13) (2001), 2517-2535.

\bibitem{ali} Uzer A., Multiplicative type complex calculus as an alternative to the classical calculus, Computers and Mathematics with Applications, 60 (2010) 2725–2737.

\bibitem{bashirov2} Bashirov A., Riza M., On complex multiplicative differentiation, TWMS Journal of Applied and Engineering Mathematics, 1(1) (2011), 75-85.

\bibitem{yazici} Yazici M., Selvitopi H., Numerical methods for the multiplicative partial differential equations, Open Math., 15 (2017), 1344–1350.

\bibitem{boruah} K. Boruah, B. Hazarika, Some Basic Properties of Bigeometric Calculus and its Applications in Numerical Analysis, Afrika Matematica, 32 (2021), 211-227.

\bibitem{aniszewska} Aniszewska, D., Multiplicative Runge-Kutta Methods. Nonlinear Dynamics, 50 (2007), 262-272.

\bibitem{bashirov3} Bashirov A., Mısırlı E., Tandogdu Y., Ozyapıcı A., On modeling with multiplicative differential equations, Appl. Math. J. Chinese Univ., 26(4) (2011), 425-438.

\bibitem{yalcin} Yalçın N., Celik E., Solution of multiplicative homogeneous linear differential equations with constant exponentials, New Trends in Mathematical Sciences, 6(2) (2018), 58-67.

\bibitem{waseem} Waseem M., Noor M.A., Shah F.A., Noor K.I., An efficient technique to solve nonlinear equations using multiplicative calculus, Turkish Journal of Mathematics, 42 (2018), 679-691.

\bibitem{bashirov} Bashirov A.E., Kurpınar E.M., Özyapıcı A., Multiplicative calculus and its applications. J. Math. Anal. Appl., 337 (2008), 36–48.

\bibitem{emrah1} Gulsen T., Yilmaz E., Goktas S., Multiplicative Dirac system. Kuwait J.Sci., 49(3) (2022), 1-11.

\bibitem{emrah2} Goktas S., Kemaloglu H., Yilmaz E., Multiplicative conformable fractional Dirac system. Turk. J. Math., 46 (2022), 973–990.

\bibitem{emrah3} Goktas S., Yilmaz E., Yar A.C., Multiplicative derivative and its basic properties on time
scales. Math. Meth. Appl. Sci., 45 (2022), 2097-2109.

\bibitem{svetlin} Georgiev S.G., Zennir K., Multiplicative Differential Calculus (1st ed.), Chapman and Hall/CRC., New York, 2022.

\bibitem{svetlin2}Georgiev S.G., Multiplicative Differential Geometry (1st ed.), Chapman and Hall/CRC., New York, 2022.

\bibitem{svetlin3} Georgiev S.G., Zennir K., Boukarou A., Multiplicative Analytic Geometry (1st ed.), Chapman and Hall/CRC., New York, 2022.

\bibitem{karacan} Nurkan S.K., Gurgil I., Karacan M.K., Vector properties of geometric calculus. Math. Meth. Appl. Sci., (2023), 1–20.

\bibitem{evren} Aydin M.E., Has A., Yilmaz B., A non-Newtonian approach in differential geometry of curves: multiplicative rectifying curves. ArXiv, (2023),  
https://doi.org/10.48550/arXiv.2307.16782

\bibitem{bio}Watson J.D., Baker T., Stephen P.B. 
et al., Molecular Biology of the Gene. Pearson Publishing, London, 2013.

\bibitem{bil}Hughes J., Dam A.V., McGuire M., Sklar D. et al., Computer Graphics: Principles and Practice, Addison-Wesley Professional Publishing, New York, 2013.

\bibitem{hav} Sadraey M.H., Aircraft Design: A Systems Engineering Approach. Wiley, New York, 2013

\bibitem{izumuya} Izumuya, S., Takeuchi, N., New special curves and developable surfaces, Turkish Journal of Mathematics, 28 (2004), 153-163.

\bibitem{takahashi} Takahashi, T., Takeuchi, N., Clad helices and developable surface, Tokyo Gakugei University Bulletin, Natural Science, 66 (2014), 1-9.

\bibitem{aykut} Yılmaz, B., Has, A., New Approach to Slant Helix. International Electronic Journal of Geometry, 12 (2019), 111-115.

\bibitem{mahmut} Mak, M., Framed clad helices in Euclidean 3-space, Filomat, 37(28) (2023), 9627–9640.

\bibitem{kaya} Kaya, S., Ateş, O., Gök, İ., Yaylı, Y., Timelike clad helices and developable surfaces in Minkowski 3-space, Rend. Circ. Mat. Palermo, II. Ser 68 (2019), 259–273.

\end{thebibliography}
\end{document}